\documentclass[11pt, letterpaper]{amsart}

\usepackage{fullpage}

\usepackage{setspace}
\usepackage[pdftex]{hyperref}
\usepackage{enumerate}
\usepackage{graphicx}

\usepackage[all,cmtip]{xy}
\usepackage{verbatim}

\usepackage{amsmath}
\usepackage{amssymb}
\usepackage{amsthm}
\usepackage{mathrsfs}
\newcommand{\scr}[1]{\mathscr{#1}}

\newcommand{\bQ}{\mathbb{Q}}
\newcommand{\N}{\mathbb{N}}

\renewcommand{\a}{\mathfrak{a}}
\newcommand{\m}{\mathfrak{m}}
\newcommand{\I}{\mathcal{I}}
\renewcommand{\O}{\mathcal{O}}

\newcommand{\tensor}{\otimes}   % tensor product
\DeclareMathOperator{\Spec}{{Spec}}

\newcommand{\tld}{\widetilde }

\newcommand{\sI}{\scr{I}}

\newcommand{\sL}{\scr{L}}
\newcommand{\sM}{\scr{M}}

\newcommand{\sP}{\scr{P}}
\newcommand{\sQ}{\scr{Q}}
\newcommand{\sR}{\scr{R}}

\theoremstyle{plain}

\newtheorem*{theorem*}{Theorem}
\newtheorem{theorem}{Theorem}[section]
\newtheorem{proposition}[theorem]{Proposition}
\newtheorem{lemma}[theorem]{Lemma}

\newtheorem{corollary}[theorem]{Corollary}

\theoremstyle{remark}
\newtheorem{remark}[theorem]{Remark}

\theoremstyle{definition}
\newtheorem{definition}[theorem]{Definition}

\newtheorem{notation}[theorem]{Notation}

\newtheorem*{question*}{Question}

\newtheorem{example}[theorem]{Example}
\newtheorem*{example*}{Example}

\begin{document}
\numberwithin{equation}{theorem}

\title{On the number of compatibly Frobenius split subvarieties, prime $F$-ideals, and log canonical centers}
\author{Karl Schwede}
\address{Department of Mathematics, University of Michigan, Ann Arbor, Michigan  48109}
\email{kschwede@umich.edu}
\author{Kevin Tucker}
\address{Department of Mathematics, University of Michigan, Ann Arbor, Michigan  48109}
\email{kevtuck@umich.edu}
\subjclass[2000]{13A35, 14B05, 14J17}
\thanks{The first author was partially supported as an NSF postdoc.  The second author was partially supported by the NSF under grant DMS-0502170.}
\begin{abstract}
Let $X$ be a projective Frobenius split variety over an algebraically closed field with splitting $\theta : F_* \O_X \rightarrow \O_X$.  In this paper we give a sharp bound on the number of subvarieties of $X$ compatibly split by $\theta$.  In particular, suppose $\sL$ is a sufficiently ample line bundle on $X$ (for example, if $\sL$ induces a projectively normal embedding) with $n = \dim H^0(X, \sL)$. We show that the number of $d$-dimensional irreducible subvarieties of $X$ that are compatibly split by $\theta$ is less than or equal to ${n \choose d+1}$.  This generalizes a well known result on the number of closed points compatibly split by a fixed splitting $\theta$.  Similarly, we give a bound on the number of prime $F$-ideals of an $F$-finite $F$-pure local ring.

Compatibly split subvarieties are closely related to log canonical centers.  Our methods apply in any characteristic, and so we are also able to bound the possible number of log canonical centers of a log canonical pair $(X, \Delta)$ passing through a closed point $x \in X$.  Specifically, if $n$ is the embedding dimension of $X$ at $x$, then the number of $d$-dimensional log canonical centers of $(X, \Delta)$ through $x$ is less than or equal to ${n \choose d}$.
\end{abstract}

\maketitle
%\tableofcontents

\spacing{1.3}

\section{Introduction}
%Collections of prime ideals whose corresponding subschemes satisfy strong tranversality conditions appear naturally in several contexts of commutative algebra and algebraic geometry.  In this paper, we study collections which
In this paper, we give a sharp bound on the size of a collection of prime ideals whose corresponding subschemes satisfy certain strong transversality conditions.  Such collections appear naturally in several contexts in both commutative algebra and algebraic geometry.  We will focus on the following three instances, which are all quite closely related (see \cite{SchwedeCentersOfFPurity} and \cite{SchwedeFAdjunction}):
\begin{itemize}
\item{}  The ideals of compatibly split subvarieties which appear in the study of Frobenius split varieties.  See, for example, \cite{MehtaRamanathanFrobeniusSplittingAndCohomologyVanishing} and \cite{BrionKumarFrobeniusSplitting}.
\item{}  The ideals of log canonical centers of a log canonical pair $(X, \Delta)$ which appear in the minimal model program and its applications.  See, for example, \cite{KawamataFujita}, \cite{FujinoIndicesOfLogCanonical} and \cite{AmbroQuasiLogVarieties}.
\item{}  The prime annihilators of $F$-stable submodules of $H^{\dim R}_{\m}(R)$, when $(R, \m)$ is an $F$-finite $F$-pure local ring, form a subset of such a collection.  These objects have been studied in several contexts, including tight closure theory.  See, for example, \cite{SmithTestIdeals}, \cite{EnescuHochsterTheFrobeniusStructureOfLocalCohomology}, \cite{SharpGradedAnnihilatorsOfModulesOverTheFrobeniusSkewPolynomialRing} and \cite{HartshorneSpeiserLocalCohomologyInCharacteristicP}.
\end{itemize}

Our main technical result, which we apply to each of the three contexts above, is the following:
\vskip 6pt
\hskip -12pt
{\bf Theorem \ref{ThmTechnicalResult}.}
{\it Let $\sQ$ be a collection of prime ideals in an excellent local ring $(R, \m)$ of embedding dimension $n$.  Suppose that the set of ideals
\[
\sI = \left\{ \left. \quad \bigcap\limits_{Q \in \alpha} Q \quad \right|  \quad \alpha \mbox{ is a finite subset of } \sQ  \quad \right\}
\]
is closed under sums.  Then the number of primes $Q \in \sQ$ such that $R/Q$ has dimension $d$ is less than or equal to ${n \choose d}$.}
\vskip 6pt
\hskip -12pt The above numerical bound also holds for certain collections of ideals whose elements are not necessarily prime.  This generalization allows us to reduce to the case where $R$ is a complete regular local ring.
%(which allows us to reduce to the case where $R$ is a complete regular local ring).

For projective varieties in positive characteristic, the existence of a Frobenius splitting has strong geometric and cohomological implications.  This observation has been particularly useful in answering numerous questions arising in representation theory, \cite{BrionKumarFrobeniusSplitting}.  Compatibly split subvarieties form a key part of this theory.  Recently, there have been several results on the finiteness of compatibly Frobenius split subvarieties; see \cite[Theorem 5.8]{SchwedeFAdjunction} and \cite{KumarMehtaFiniteness}.
We use Theorem \ref{ThmTechnicalResult} to prove the following result which gives a sharp bound on the number of compatibly split subvarieties:
\vskip 6pt
\hskip -12pt
{\bf Theorem \ref{ThmCompatiblySplitProjective}.}
{\it Suppose that $X$ is a projective variety over an $F$-finite field $k$ of characteristic $p > 0$.  Further, suppose that $\sL$ is an ample line bundle on $X$ with associated section ring $S = \oplus_{n \geq 0} H^0(X, \sL^{n})$.  Let $n$ be the embedding dimension of $S$ at the irrelevant ideal $S_+$.  If $\theta : F^e_* \O_X \rightarrow \O_X$ is a splitting of the $e$-iterated Frobenius, then there are at most ${n \choose d+1}$ irreducible $d$-dimensional subschemes of $X$ which are compatibly split with $\theta$.
}
\vskip 6pt
\hskip -12pt
Note that, if $k$ is an algebraically closed field and $\sL$ is a very ample line bundle which induces a projectively normal embedding, then $n = \dim H^0(X, \sL)$; see Remark \ref{RemEmbDimOfProjNormal}.
This result also generalizes a well known bound on the number of closed points compatibly split by a fixed Frobenius splitting.  Specifically, in the notation of Theorem \ref{ThmCompatiblySplitProjective}, suppose that $Y$ is the reduced scheme corresponding to a union of $\theta$-compatibly split closed points in $X$.  Since $H^{0}(X, \sL) \to H^{0}(Y, \sL)$ is surjective by
\cite[Theorem 1.2.8 (ii)]{BrionKumarFrobeniusSplitting}, there can be at most $\dim H^{0}(X,\sL)$ distinct points in $Y$. For any ample line bundle $\sL$ on $X$ with $n = \dim H^0(X, \sL)$, it would be interesting to know if the number of $\theta$-compatibly split subvarieties of dimension $d$ is at most ${n \choose d+1}$.

Next, recall that an $F$-ideal of a local ring $(R,\m,k)$ of positive characteristic is an annihilator of an $F$-stable submodule of $H^{\dim R}_{\m}(R)$.  From the point of view of tight closure theory, $F$-stable submodules of $H^{\dim R}_{\m}(R)$ and their annihilators are natural objects of study; see \cite{SmithTestIdeals}.  For example, if $R$ is normal, local and Gorenstein, then the largest proper $F$-stable submodule of $H^{\dim R}_{\m}(R)$ is the (finitistic) tight closure of 0, and its annihilator is the test ideal of $R$.  There have been several recent papers studying when the set of $F$-ideals is finite; see for example \cite{SharpGradedAnnihilatorsOfModulesOverTheFrobeniusSkewPolynomialRing} and \cite{EnescuHochsterTheFrobeniusStructureOfLocalCohomology}.  We use Theorem \ref{ThmTechnicalResult} to deduce the following partial generalization of these recent results:
\vskip 6pt
\hskip -12pt
{\bf Theorem \ref{ThmBoundOnFIdeals}.}
{\it Suppose that $(R, \m)$ is an $F$-finite local ring of embedding dimension $n$.  Further, suppose that $\Spec R$ is $F$-pure.  Then the set of prime $d$-dimensional $F$-ideals of $R$ is less than or equal to ${n \choose d}$.}
\vskip 6pt
\hskip -12pt Since every $F$-ideal in an $F$-pure ring can be written as an intersection of prime $F$-ideals, this can also be used to give a bound on the number of arbitrary $F$-ideals.
%Theorem \ref{ThmBoundOnFIdeals} (and also Theorem \ref{ThmBoundOnSubmodules}) should be viewed as a partial generalization of some of the main results of \cite{SharpGradedAnnihilatorsOfModulesOverTheFrobeniusSkewPolynomialRing} and \cite{EnescuHochsterTheFrobeniusStructureOfLocalCohomology}.

Finally, we have the following result in characteristic zero (which also follows from Theorem \ref{ThmTechnicalResult}):
\vskip 6pt
\hskip -12pt
{\bf Theorem \ref{ThmBoundOnLCCenters}.}
{\it Suppose that $(X, \Delta)$ is a log canonical pair and that $x \in X$ is a point with embedding dimension $n$.  Then the number of $d$-dimensional log canonical centers of $(X, \Delta)$ which contain $x$ is less than or equal to ${n \choose d}$.}
\vskip 6pt
\hskip -12pt
Similar techniques can be used to bound the number of log canonical centers of a log Calabi-Yau pair $(X, \Delta)$; see Remark \ref{logcalabiyau}.

We should also mention that, in all of the situations we consider, the given bounds are sharp.  In particular, some variant of Example \ref{variablesexample} can occur in each of these contexts.
%It may be that similar techniques can be used to bound the number of log canonical centers of a pair $(X, \Delta)$ which is log Calabi-Yau or some variant of log Fano (at least when $X$ is projective).  %In particular, it may be that an appropriate section ring $S$ of $X$ is log canonical and the log canonical centers of $(X, \Delta)$ appear as log canonical centers on that section ring.
%However, we will not explore this possibility here.

{\it Acknowledgements: }

The first author began thinking about this question in the fall of 2008, when Jesper Funch Thomsen asked him if he knew how to bound the number of subvarieties compatibly split by a fixed Frobenius splitting of a projective variety.  The authors would also like to thank Anurag Singh for several very useful conversations as well as Sam Payne, Karen Smith, S\'andor Kov\'acs, Alan Stapledon, and Anne Shiu for discussions.  The authors would also like to thank the referee for several very helpful suggestions.  The authors spent many hours working on this problem at MSRI in the Winter of 2009.

\section{Pseudo-prime systems of ideals}

All rings in this paper will be assumed to be commutative with unity, Noetherian, and excellent.  All schemes will be assumed to be Noetherian and separated.  By a variety over a field $k$, we mean a separated integral scheme of finite type over $\Spec k$.

Let $R$ be a local ring, and $Q$ a proper ideal of $R$.  Recall that the \emph{dimension} or \emph{coheight} of $Q$ is simply the dimension of the local ring $R/Q$.  We say $Q$ is \emph{equidimensional} if all of the minimal primes of $R/Q$ have the same dimension.

Note that if a prime ideal of an excellent local ring $R$ is extended to the completion $\hat{R}$, it may no longer be prime.
It is because of this issue, and the fact that we complete in the proof of our main technical result, Theorem \ref{ThmTechnicalResult}, that we now introduce the notion of a pseudo-prime system of ideals.  These should be thought of as a generalization of a set of prime ideals, and are meant to capture properties of collections of prime ideals which are preserved under completion; see Proposition \ref{change}.

\begin{definition}
Let $R$ be a local ring, and $\sQ$ a set of ideals in $R$.  We say that $\sQ$ is a \emph{pseudo-prime system} if the following two conditions hold:
\begin{enumerate}
\item
Every $Q \in \sQ$ is proper, radical, and equidimensional.
\item
If $Q_{1}, Q_{2} \in \sQ$ and some minimal prime of $Q_{1}$ is contained in a minimal prime of $Q_{2}$, then $Q_{1} \subseteq Q_{2}$.
\end{enumerate}
We shall denote by $e(R,\sQ,d)$ the number of ideals in $\sQ$ of dimension $d$.
\end{definition}

\begin{remark}
If $\sQ$ is a pseudo-prime system,  then any subset of $\sQ$ is also a pseudo-prime system.  Further, it follows from (2) that distinct elements of $\sQ$ cannot have a minimal prime in common. Note that any set of prime ideals is a pseudo-prime system.
\end{remark}

\begin{example}[Behavior of dimension in chains of $\sQ$] Suppose $\sQ$ is a pseudo-prime system, and
$Q_{1} \subsetneq Q_{2} \subsetneq \ldots \subsetneq Q_{r}$ is a chain of ideals in $\sQ$.  Condition (2) guarantees that the dimension of $Q_{i+1}$ is strictly less than the dimension of $Q_{i}$.  Indeed, suppose that $\dim(Q_{i+1}) = \dim(Q_{i})$ for some $i$.  Let $P_{i+1}$ be a minimal prime of $Q_{i+1}$.  Since $Q_{i} \subset P_{i+1}$, there exists a minimal prime $P_{i}$ of $Q_{i}$ with $P_{i} \subset P_{i+1}$.  It follows from $\dim(R/P_{i}) = \dim(Q_{i}) = \dim(Q_{i+1}) = \dim(R/P_{i+1})$ that $P_{i} = P_{i+1}$.  From (2), we see that $Q_{i+1} = Q_{i}$, which is a contradiction.
\end{example}

We now address the stability of pseudo-prime systems under various algebraic operations.  Again, note that part (i.) is the essential point which will allow us to reduce to the case of a complete ring in the proof of our main theorem.

\begin{proposition}
\label{change}
Suppose $\sQ$ is a pseudo-prime system in an excellent local ring $(R,\m,k)$.
\begin{enumerate}[(i.)]
\item
Let $\hat R$ denote the $\m$-adic completion of $R$.
The set of ideals $\sQ \hat R := \{ \, Q \hat R \, | \, Q \in \sQ \, \}$ of $\hat R$ is a pseudo-prime system, and the map $ Q \mapsto Q \hat R$ gives a bijection between $\sQ$ and $\sQ \hat R$.  In particular, we have $e(R,\sQ,d) = e(\hat R, \sQ \hat R, d)$ for all $d$.
\item
If $P$ is a prime ideal in $R$, then the set of ideals $\sQ R_{P} := \{ \, Q R_{P} \, | \, Q \in \sQ \mbox{ and } Q \subseteq P \, \}$ of $R_{P}$ is a pseudo-prime system, and the map $Q \mapsto Q R_{P}$ gives a bijection between $\{ \, Q \in \sQ \, | \, Q \subseteq P \, \}$ and $\sQ R_{P}$.  In particular, if $p$ is the dimension of $P$, we have that $e(R_{P}, \sQ R_{P}, d-p)$ equals the number of ideals $Q \in \sQ$ with dimension $d$ such that $Q \subseteq P$.
\item
If $I$ is any ideal of $R$, then the set of ideals $\sQ/I := \{ \, Q/I \, | \, Q \in \sQ \mbox{ and } I \subseteq Q \, \}$ of $R/I$ is a pseudo-prime system, and the map $Q \mapsto Q/I$ gives a bijection between $\{ \, Q \in \sQ \, | \, I \subseteq Q \, \}$ and $\sQ/I$.  In particular, we have that $e(R/I, \sQ/I, d)$ equals the number of ideals $Q \in \sQ$ with dimension $d$ such that $I \subseteq Q$.
\item
If $\phi: S \to R$ is a surjective morphism of local rings, then the set of ideals $\phi^{-1}(\sQ) := \{ \, \phi^{-1}(Q) \, | \, Q \in \sQ \, \}$ is a pseudo-prime system, and the map $Q \mapsto \phi^{-1}(Q)$ gives a bijection between $\sQ$ and $\phi^{-1}(\sQ)$.  In particular, we have $e(R,\sQ,d) = e(S, \phi^{-1}(\sQ), d)$ for all $d$.

\end{enumerate}
\end{proposition}

\begin{proof}
{\itshape (i.)} Recall that the completion of a reduced equidimensional excellent local ring remains reduced and equidimensional.  See, for example, 7.8.3 (vii) and  (x) in \cite{EGAIV2}. Thus, applying this fact to $R/Q$ for $Q \in \sQ$, we see that $\sQ \hat R$ satisfies (1).  To verify (2), suppose that $Q_{i} \in \sQ$ and $P_{i}$ is a minimal prime of $Q_{i} \hat R$ for $i = 1,2$.  If $P_{1} \subseteq P_{2}$, then we have $P_{1} \cap R \subseteq P_{2} \cap R$.  But $P_{i} \cap R$ is necessarily a minimal prime of $Q_{i}$ for $i = 1,2$.  Thus, it follows that $Q_{1} \subseteq Q_{2}$, whence $Q_{1} \hat R \subseteq Q_{2} \hat R$.  Finally, since $R \to \hat R$ is faithfully flat, we have $I \hat R \cap R = I$ for any ideal $I$ in $R$.  Thus, the map $Q \hat R \mapsto Q \hat R \cap R$ is an inverse to the map $Q \mapsto Q \hat R$  between $\sQ$ and $\sQ \hat R$.

{\itshape (ii.)}  Set $\phi : R \rightarrow R_P$ to be the canonical map.  Without loss of generality, we may assume $Q \subseteq P$ for all $Q \in \sQ$.    It follows immediately that $Q R_{P}$ is a proper radical ideal in $R_{P}$. Since $R$ is excellent, it is catenary, and it follows that $Q R_{P}$ is equidimensional.  Thus, $\sQ R_{P}$ satisfies (1).  To verify (2), suppose that $Q_{i} \in \sQ$ and $P_{i}$ is a minimal prime of $Q_{i} R_{P}$ for $i = 1,2$.  If $P_{1} \subseteq P_{2}$, then we have $\phi^{-1}(P_1) \subseteq \phi^{-1}(P_2)$.  But $\phi^{-1}(P_{i})$ is necessarily a minimal prime of $Q_{i}$ for $i=1,2$.  Thus, it follows that $Q_{1} \subseteq Q_{2}$, whence $Q_{1} R_{P} \subseteq Q_{2} R_{P}$.  Finally, if $Q_{1}', Q_{2}' \in \sQ$ %and $Q_{1}', Q_{2}' \subseteq P$
with $Q_{1}' R_{P} = Q_{2}' R_{P}$, then $Q_{1}' R_{P}$ and $Q_{2}' R_{P}$ certainly have a minimal prime in common, and the previous argument implies that $Q_{1}' = Q_{2}'$.  Thus, the assignment $Q \mapsto Q R_{P}$ gives an injective, and hence also bijective, map between $\sQ$ and $\sQ R_{P}$.

{\itshape (iii.) and (iv.)} Proofs of these statements follow easily from applications of the correspondence theorem, and are left for the reader to produce.
\end{proof}

\section{An intersection condition}

We now state an intersection condition for pseudo-prime systems.  We suggest the reader first think about this definition in the case that $\sQ$ is simply a collection of prime ideals.

\begin{definition} Let $(R,\m,k)$ be a local ring of dimension $n$, and let $\sQ$ be a pseudo-prime system in $R$. We say $\sQ$ is an \emph{intersection compatible system} (or simply a \emph{compatible system}) if for all finite subsets $\alpha_{1}, \ldots, \alpha_{r}$ of $\sQ$, there exists a finite subset $\beta$ of $\sQ$ with
\[
\sum_{i=1}^{r}  \left( \bigcap\limits_{Q \in \alpha_{i}} Q \right) = \bigcap\limits_{Q \in \beta} Q.
\]
In other words, the set of ideals
\[
\sI = \left\{ \left. \quad \bigcap\limits_{Q \in \alpha} Q \quad \right|  \quad \alpha \mbox{ is a finite subset of } \sQ  \quad \right\}
\]
is closed under sums.
\end{definition}

\begin{remark}
The above condition can also be phrased geometrically in the following way. The ideals in $\sQ$ correspond to reduced equidimensional subschemes of the affine scheme $\Spec(R)$.  The collection $\sQ$ is intersection compatible if the set of finite unions of these subschemes is closed under scheme-theoretic intersection.
\end{remark}

\begin{example}
\label{variablesexample}
If $R = k[[ x_{1}, \ldots, x_{n} ]]$ is the ring of formal power series over $k$ with variables $x_{1}, \ldots, x_{n}$, the collection $\sQ$ of prime ideals generated by subsets of the variables is intersection compatible.  Note that, in this example, there are precisely ${n \choose d}$ prime ideals in $\sQ$ of dimension $d$.
\end{example}

\begin{example}[Non-Example]
If $R = k[[x,y]]$, then the collection $\{ \langle x,y \rangle, \langle x \rangle, \langle y \rangle, \langle x + y \rangle \}$ is not intersection compatible.  Although this set of primes is closed under pairwise sum, we have that
\[
\left( \langle x \rangle \cap \langle y \rangle \right) + \langle x + y \rangle = \langle xy, x+y \rangle = \langle x^{2}, x + y \rangle
\]
is not even reduced.
\end{example}

\begin{example}[Chains of prime ideals]
If $Q_{1} \subsetneq Q_{2} \subsetneq \ldots \subsetneq Q_{r}$ is a chain of prime ideals in a local ring $R$, then one can easily verify that $\sQ = \{ Q_{1}, \ldots, Q_{s} \}$ is an intersection compatible system.  In particular, it is easy to construct examples where the subschemes corresponding to the ideals in $\sQ$ have arbitrarily singular components.
\end{example}

\begin{proposition}
\label{compat}
Suppose $\sQ$ is a pseudo-prime system in an excellent local ring $(R,\m,k)$. If $\sQ$ is intersection compatible, then so are $\sQ \hat R$, $\sQ R_{P}$, $\sQ / I$, and $\phi^{-1}(\sQ)$ as described in Proposition \ref{change}.
\end{proposition}

\begin{proof}
If $T$ is any $R$-algebra and $I_{1}, I_{2}$ are ideals in $R$, we always have $I_{1} T + I_{2} T = (I_{1}+I_{2})T$.  In addition, when $T$ is flat over $R$, we have $I_{1}T \cap I_{2}T = (I_{1} \cap I_{2})T$.  Since $\hat R$ and $R_{P}$ are flat over $R$, it follows immediately that $\sQ \hat R$ and $\sQ R_{P}$ are intersection compatible.

Similarly, if $I_{1}, I_{2}$ are ideals in $R$, and $\phi:S \to R$ is any morphism, we always have $\phi^{-1}(I_{1}) \cap \phi^{-1}(I_{2}) = \phi^{-1}(I_{1} \cap I_{2})$.  In addition, when $\phi$ is surjective, we have $\phi^{-1}(I_{1}) + \phi^{-1}(I_{2}) = \phi^{-1}(I_{1}+I_{2})$.  It follows that $\phi^{-1}(\sQ)$ is intersection compatible.

Finally, if $I_{1}, I_{2}$ are ideals in $R$ containing an ideal $I$, then $I_{1} / I + I_{2}/I = (I_{1}+I_{2})/I$ and $(I_{1}/I) \cap (I_{2}/I) = (I_{1} \cap I_{2}) / I$.  Thus, $\sQ / I$ is intersection compatible as well.
\end{proof}

\section{The main technical result}

We first need the following elementary result.

\begin{lemma}
\label{ebound}
Suppose $(R,\m,k)$ is a local ring, and $I_{1}, I_{2}$ are two ideals in $R$ such that $I_{1} + I_{2} = \m$.
Then
 \[
 \dim_{k} \left( \m/I_{1} \Big/ \left(\m/I_{1} \right)^{2} \right)
 +  \dim_{k} \left( \m/I_{2} \Big/ \left(\m/I_{2} \right)^{2} \right)
\leq  \dim_{k} \left( \m / \m^{2} \right).
\]
\end{lemma}

\begin{proof}
Consider the diagonal mapping
\begin{eqnarray*}
\delta : \m & \to & ( \m / I_{1}) \oplus (\m / I_{2}) \\
m & \mapsto & (m + I_{1}, m + I_{2}).
\end{eqnarray*}
Using that $I_{1}+I_{2} = \m$, it is easy to check that this map is surjective.  Indeed, suppose $e_{1}, e_{2} \in \m$.  We can write
\[ e_{1} = a_{1} + b_{1} \qquad e_{2} = a_{2} + b_{2}
\]
where $a_{1},a_{2} \in I_{1}$ and $b_{1},b_{2} \in I_{2}$.  Then we have
\[
\delta(a_{2} + b_{1}) = (b_{1} + I_{1}, a_{2} + I_{2}) = (e_{1} + I_{1}, e_{2} + I_{2}).
\]
Thus, $\delta$ is surjective.
Since the kernel of $\delta$ is manifestly equal to $I_{1} \cap I_{2}$, we have an induced isomorphism
\[
\m / \left( I_{1} \cap I_{2} \right) \cong ( \m / I_{1}) \oplus (\m / I_{2}).
\]
Thus, it follows that
\[
 \dim_{k} \left( \m/I_{1} \Big/ \left(\m/I_{1} \right)^{2} \right)
 +  \dim_{k} \left( \m/I_{2} \Big/ \left(\m/I_{2} \right)^{2} \right)
=  \dim_{k} \left( \m/(I_{1}\cap I_{2}) \Big/ \left(\m/(I_{1}\cap I_{2}) \right)^{2} \right)
\]
and the conclusion now follows since
\[
\dim_{k} \left( \m/(I_{1}\cap I_{2}) \Big/ \left(\m/(I_{1}\cap I_{2}) \right)^{2} \right) \leq
 \dim_{k} \left( \m / \m^{2} \right).
\]
\end{proof}

We now prove our main result.

\begin{theorem}
\label{ThmTechnicalResult} Let $E(d,n)$ be the supremum of the numbers $e(R,\sQ,d)$, where $R$ and $\sQ$ vary over all excellent local rings $(R,\m,k)$ with embedding dimension $n = \dim_{k}(\m / \m^{2})$ and all intersection compatible systems $\sQ$ in $R$.  Then $E(d,n) = {n \choose d}$.  In particular, for any intersection compatible system $\sQ$ of \emph{prime} ideals in such a ring $R$, $e(R, \sQ, d) \leq {n \choose d}$.
\end{theorem}

\begin{proof}
From example \ref{variablesexample}, it follows that $E(d,n) \geq {n \choose d}$.  Suppose, by way of contradiction, we have $E(d,n) > {n \choose d}$ for some values of $n$ and $d$.  We may assume that $(d,n)$ is the smallest pair of natural numbers with this property, where $\N^{2}$ is ordered lexicographically.  In other words, we have $E(d',n') \leq {n' \choose d'}$ whenever $d' < d$, or $d' = d$ and $n' < n$.

Let $\sQ$ be an intersection compatible system in an excellent local ring $(R,\m,k)$ with embedding dimension $n = \dim_{k}(\m / \m^{2})$ such that $e(R,\sQ,d) > {n \choose d}$.  If $\hat R$ denotes the $\m$-adic completion of $R$,  it follows from the structure theory of complete local rings that there is a surjective morphism $\phi: S \to \hat R$ where $S$ is a complete regular local ring of dimension $n$.  See, for example, Theorem 7.16 in \cite{Eisenbud}. By Propositions \ref{change} and \ref{compat}, we have that $\sQ \hat R$ and $\phi^{-1}(\sQ \hat R)$ are again compatible systems, and $e(R, \sQ, d) = e(\hat R, \sQ \hat R, d) = e(S, \phi^{-1}(\sQ \hat R), d)$.  Replacing $R$ by $S$ and $\sQ$ by $\phi^{-1}(\sQ \hat R)$, we may assume that $R$ is a complete regular local ring of dimension $n$.

Suppose first that $\m \not \in \sQ$, and consider $\sum_{Q \in \sQ} Q$. Since $R$ is Noetherian, there exist $Q_{1}, \ldots, Q_{r} \in \sQ$ with
$\sum_{Q \in \sQ} Q = \sum_{i=1}^{r} Q_{i}$.  As $\sQ$ is intersection compatible, there is some $Z \in \sQ$ with $ \sum_{i=1}^{r} Q_{i} \subseteq Z$.  Thus, we have $Z \in \sQ$ and $Q \subseteq Z$ for all $Q \in \sQ$.  Since $\m \not\in \sQ$, we have $z := \dim Z \geq 1$.  Let $P_{Z}$ be a minimal prime of $Z$.  By Proposition \ref{change}, we have $e(R, \sQ, d) = e(R_{P_{Z}}, \sQ R_{P_{Z}}, d-z)$.  Since $\sQ R_{P_{Z}}$ is compatible by Proposition \ref{compat}, it follows from the minimality of $(d,n)$ that
\[
e(R,\sQ,d) = e(R_{P_{Z}}, \sQ R_{P_{Z}}, d-z) \leq {n-z \choose d - z } \leq {n \choose d},
\]
which is a contradiction.  Note that we have used the fact that $R$ is regular to control the embedding dimension of $R_{P_{Z}}$.

Thus, we may assume that $\m \in \sQ$.  Since $\m$ is the only prime ideal of dimension zero in $R$, it follows that $e(R,\sQ,0) = {n \choose 0} = 1$, so we must have $d \geq 1$.  Thus, the collection $\sQ \setminus \{ \m \}$  is nonempty and must have a maximal element $Y$, as $R$ is Noetherian.  In other words, we have $\m \neq Y \in \sQ$ and $\{\, Q \in \sQ \, | \, Y \subseteq Q \, \} = \{ Y,\m \}$.   Set $y := \dim Y \geq 1$, and let $P_{Y}$ be a minimal prime of $Y$.  Let $\sP = \{ \, Q \in \sQ \, | \, Q \subseteq P_{Y} \, \}$.  By Propositions \ref{change} and \ref{compat}, we have $e(R, \sP,d) = e(R_{P_{Y}}, \sQ R_{P_{Y}}, d-y)$ and $\sQ R_{P_{Y}}$ is compatible.  By the minimality of $(d,n)$, it follows that $e(R,\sP,d) \leq {n-y \choose d -y} \leq {n-1 \choose d-1}$.  Again, we have used that $R$ is regular to control the embedding dimension of $R_{P_Y}$.  Hence, we must have
\[
e(R, \sQ \setminus \sP, d) = e(R, \sQ, d) - e(R, \sP, d) > {n \choose d} - {n-1 \choose d-1} = {n-1 \choose d}.
\]

Let $Q_{1}, \ldots, Q_{{n-1 \choose d} + 1}$ be distinct elements of $\sQ \setminus \sP$ of dimension $d$, and set $I = Q_{1} \cap \cdots \cap Q_{{n-1 \choose d} + 1}$.   If $\sR = \{ \, Q \in \sQ \, | \, I \subseteq Q \, \}$,  we have by construction $e(R, \sR, d) > {n-1 \choose d}$.  By Propositions \ref{change} and \ref{compat}, $\sQ/I$ is again a compatible system, and $e(R, \sR, d) = e(R/I, \sQ/I, d)$.
From the minimality of $(d,n)$, it follows that $ \dim_{k} \left( \m/I \Big/ \left(\m/I \right)^{2} \right) = n$.  Since $\dim_{k} \left( \m/Y \Big/ \left(\m/Y \right)^{2} \right) \geq \dim(R/Y) = \dim Y = y \geq 1$,
Lemma \ref{ebound} implies that $I + Y \neq \m$.  However, since $\sQ$ is compatible, $I + Y$ is equal to the intersection of all $Q \in \sQ$ with $I+Y \subseteq Q$.
By the maximality of $Y$, we must have $I + Y = Y$, i.e. $I \subseteq Y$.

Since $I \subseteq Y \subseteq P_{Y}$, there is a minimal prime $P$ of $I$ with $P \subseteq P_{Y}$.  From the definition of $I$, it follows that there is some $Q_{i}$ having $P$ as a minimal prime.  Since $\sQ$ is a pseudo-prime system, we conclude that $Q_{i} \subseteq Y$.  But this is absurd, since $Y \subseteq P_{Y}$ and $Q_{i}$ was chosen so that $Q_{i} \not\subseteq P_{Y}$.
\end{proof}

\begin{corollary}
If $\sQ$ is an intersection compatible system in an excellent regular local ring $R$ of dimension $n$, then $\sQ$ is finite and there are at most $2^{n}$ ideals in $\sQ$.
\end{corollary}

\section{Compatibly split subvarieties and $F$-ideals}

We begin with some notation.  Throughout this section, we will assume that all rings and all schemes lie over a field of characteristic $p > 0$.
If $X$ is a scheme, we let $F^e : X \rightarrow X$ denote the $e$-iterated Frobenius map.  For any ring $R$ and any $R$-module $M$, we define $F^{e}_{*}M$ in accordance with the geometric notation for quasicoherent sheaves on $\Spec(R)$.  In other words, $F^e_* M$ denotes the $R$-module which is equal to $M$ as an additive group, but has the $R$-module structure $r \cdot x = r^{p^e} x$
induced by the $e$-iterated Frobenius.

\begin{definition}
We say that a ring is \emph{$F$-finite} if $F^1_* R$ is finite as an $R$-module.
\end{definition}

\begin{definition}\cite{MehtaRamanathanFrobeniusSplittingAndCohomologyVanishing}
We say that a scheme $X$ is \emph{$F$-split} if there is a $\O_X$-module splitting $\theta : F^1_* \O_X \rightarrow \O_X$ of the Frobenius map $\O_X \rightarrow F^1_* \O_X$.  A subscheme $Z \subseteq X$ is called \emph{compatibly split with $\theta$} if we have
\[
\theta(F^e_* \I_Z) \subseteq \I_Z
\]
where $\I_Z$ is the ideal sheaf of $Z$.  Note that any such $\I_Z$ is necessarily a radical ideal sheaf \cite[Proposition 1.2.1]{BrionKumarFrobeniusSplitting}.
\end{definition}

\begin{proposition}
\label{PropCompatSplitLocalRing}
Suppose that $(R,\m,k)$ is an $F$-finite local ring of embedding dimension $n = \dim_{k} (\m / \m^{2})$, and $\theta: F^e_{*} R \to R$ is a fixed surjective $R$-linear map.  Then there are at most ${n \choose d}$ prime ideals $Q$ of dimension $d$ such that $\theta(F^e_* Q) \subseteq Q$.
\end{proposition}
\begin{proof}
It is straightforward to see that the set of such ideals is an intersection compatible system. Since $F$-finite rings are excellent \cite{Kunz},
the conclusion follows from Theorem \ref{ThmTechnicalResult}.
\end{proof}

\begin{theorem}
\label{ThmCompatiblySplitProjective}
Suppose that $X$ is a projective variety over a $F$-finite field $k$ of characteristic $p > 0$.  Further, suppose that $\sL$ is an ample line bundle on $X$ with associated section ring $S = \oplus_{n \geq 0} H^0(X, \sL^{n})$.  Let $n$ be the embedding dimension of $S$ at the irrelevant ideal $S_+$.  If $\theta : F^e_* \O_X \rightarrow \O_X$ is a splitting of the $e$-iterated Frobenius, then there are at most ${n \choose d+1}$ irreducible $d$-dimensional subschemes of $X$ which are compatibly split with $\theta$.
\end{theorem}
\begin{proof}
Note that $\theta$ induces a splitting
\[
\xymatrix{
S \ar[r] &  F^e_* S \ar[r]^{\theta_S} \ar[r] & S.
}
\]
See \cite[Proposition 3.1]{SmithGloballyFRegular} and \cite[Proposition 4.10]{SmithVanishingSingularitiesAndEffectiveBounds} for further details.
For each irreducible compatibly split subscheme $Z$ of $X$, let $\I_Z$ denote the corresponding ideal sheaf.  Let $I_Z$ denote the homogeneous ideal $\oplus_{n \geq 0} H^{0}(X, \I_{Z} \tensor \sL^{n})$ in the graded ring $S$.  Note that, if $Z$ was a variety of dimension $d$, then $I_Z$ is a prime ideal of dimension $d+1$ contained in the irrelevant ideal of $S$.   It easily verified that $\theta_S(F^e_* I_Z) \subseteq I_Z$.  Thus, we can then localize at $S_{+}$, the irrelevant ideal of $S$, and apply Proposition \ref{PropCompatSplitLocalRing}.
\end{proof}

\begin{remark}
\label{RemEmbDimOfProjNormal}
If $X$ is a projective variety over a field $k$ and $\sL$ is any ample line bundle, there is a positive integer $m$ such that the section ring $S = \oplus_{n\geq 0} H^{0}(X, (\sL^{m})^{n})$ is generated in degree one over the field $S_{0} = H^{0}(X, \O_{X})
\supseteq k$.  In this situation, we have that $S_{+} / (S_{+})^{2}$ is isomorphic to $H^{0}(X, \sL^{m})$ as a vector space over $S_0$.  Thus, the embedding dimension $n$ of $S$ at the irrelevant ideal $S_{+}$ equals the dimension of $H^{0}(X, \sL^{m})$ over $S_0$.
Furthermore, when $X$ is normal and $k$ is algebraically closed, a very ample line bundle $\sM$ induces a projectively normal embedding if and only if the associated section ring is generated in degree one over $S_{0} = k$.  In this case, $n$ is equal to $\dim_k H^0(X, \sM)$.
\end{remark}

\begin{remark}
 The referee has pointed out that this result might be used to prove that certain collections of subvarieties of a given projective variety cannot be simultaneously compatibly split (due to purely numerical considerations).
\end{remark}

\begin{remark}
The first author has also recently introduced the notion of ``centers of sharp $F$-purity,'' a characteristic $p > 0$ analog of ``log canonical centers.'' We will not define centers of sharp $F$-purity here.  However, it follows from Proposition \ref{PropCompatSplitLocalRing} that if $(R, \Delta, \a_{\bullet})$ is a sharply $F$-pure triple (see \cite{SchwedeCentersOfFPurity}) where $R$ is a local ring of  embedding dimension $n$, then there are at most ${n \choose d}$ $d$-dimensional centers of sharp $F$-purity for $(R, \Delta, \a_{\bullet})$.
\end{remark}

We conclude this section with an application to annihilators of $F$-stable submodules of $H^d_{\m}(R)$ and $F$-stable submodules of $E_{R}$ (the injective hull of $k$).

\begin{definition} \cite{SmithTestIdeals}
Suppose that $M$ is an $R$-module.  A \emph{Frobenius action} on $M$ is an $R$-linear map $\rho : M \rightarrow F^e_* M$.  We say $N \subseteq M$ is \emph{$F$-stable} (with respect to the Frobenius action $\rho$) if $\rho(N) \subseteq F^e_* N \subseteq F^e_* M$.  If $M = H^{\dim R}_{\m}(R)$ and we are given the canonical Frobenius action $H^{\dim R}_{\m}(R) \rightarrow H^{\dim R}_{\m}(F^e_* R) \cong F^e_*H^{\dim R}_{\m}( R)$,
%%% IS THIS CHANGE OKAY?   IT IS, BUT WHY NOT INCLUDE BOTH?  The reason I like the other one is that I feel it is clearer what exactly the canonical map is
then an \emph{$F$-ideal of $R$} is the annihilator of any $F$-stable submodule of $M$.
\end{definition}

Suppose that $(R, \m,k)$ is a complete $F$-finite local ring, and let $E_{R}$ be the injective hull of $k$.  Then Matlis duality induces a bijection between the $R$-linear maps $\theta : F^e_* R \rightarrow R$ and Frobenius actions $\rho : E_R \rightarrow F^e_* E_R$.  In particular, $\Spec R$ is $F$-split if and only if there exists an injective Frobenius action on $E_R$.
For a fixed Frobenius action $\rho$ on $E_R$ corresponding to $\theta : F^e_* R \rightarrow R$, the $F$-stable submodules of $E_R$ (with respect to $\rho$) are in bijection with the ideals $I \subseteq R$ such that $\theta(F^e_* I) \subseteq I$.

\begin{theorem} \label{ThmBoundOnSubmodules}
Let $(R, \m,k)$ be a complete $F$-finite local ring of embedding dimension $n$, and $E_R$ the injective hull of $k$.  Suppose that we have an injective Frobenius action $\rho : E_R \rightarrow F^e_* E_R$.  Then the number of $F$-stable submodules of $E_R$ with prime annihilators of dimension $d$ is less than or equal to ${n \choose d}$.
\end{theorem}
\begin{proof}
Apply Matlis duality.
\end{proof}

\begin{theorem}  \label{ThmBoundOnFIdeals}
Suppose that $(R, \m,k)$ is an $F$-finite local ring of embedding dimension $n$.  Assume further that $\Spec R$ is $F$-split.  Then the set of prime $d$-dimensional $F$-ideals of $R$ is less than or equal to ${n \choose d}$.
\end{theorem}
\begin{proof}
It follows from \cite[Proposition 5.5]{SchwedeCentersOfFPurity} or \cite[Proof of Theorem 4.1]{EnescuHochsterTheFrobeniusStructureOfLocalCohomology} that any $F$-ideal $J \subseteq R$ satisfies $\theta(F^e_* J) \subseteq J$ for any $R$-linear map $\theta : F^e_* R \rightarrow R$.  Since $\Spec R$ is $F$-split, there exists a surjective such $\theta$.  The result then follows from Proposition \ref{PropCompatSplitLocalRing}.
\end{proof}

\begin{remark}
Suppose that $M$ is a module with an injective Frobenius action.  Then, since every $F$-stable submodule of $M$ is a sum of $F$-stable submodules with prime annihilators, one can also use Theorem \ref{ThmBoundOnSubmodules} to obtain a bound on the number of arbitrary $F$-stable submodules of $E_R$.  Similarly, in an $F$-pure ring, any $F$-ideal can be written as an intersection of prime $F$-ideals, and Theorem \ref{ThmBoundOnFIdeals} yields a bound on the number of arbitrary $F$-ideals.  See \cite[Theorem 3.10]{SharpGradedAnnihilatorsOfModulesOverTheFrobeniusSkewPolynomialRing} and \cite[Theorem 3.6]{EnescuHochsterTheFrobeniusStructureOfLocalCohomology} for additional discussion.
\end{remark}

\section{Log canonical centers}

We briefly recall the definitions of log canonical singularities and log canonical centers.  Please see \cite[Section 10.4.B]{LazarsfeldPositivity2}, \cite[Definition 1.3]{KawamataFujita}, or \cite{KollarMori} for additional background and applications.

We assume the following notation for the remainder of the section.
\begin{notation}
\label{NotationCharZeroSings}
Let $X$ be a normal variety defined over an algebraically closed field of characteristic zero.  Suppose that $\Delta$ is an effective $\bQ$-divisor on $X$ such that $K_X + \Delta$ is $\bQ$-Cartier.  We assume that $\pi : \tld X \rightarrow X$ is a log resolution of the pair $(X, \Delta)$ (see \cite[Notation 0.4]{KollarMori} or \cite[Theorem 4.1.3]{LazarsfeldPositivity1} for a definition of log resolutions) and that $\pi_* K_{\tld X} = K_X$.  Write
\begin{equation}
\label{EqnDiscrepancy}
K_{\tld X} - \pi^* (K_X + \Delta) = \sum a_i E_i
\end{equation}
where the $E_i$ are prime divisors on $\tld X$.
Note that the $E_i$ that appear on the right side are all exceptional except for those that correspond to components of the strict transform of $\Delta$.
\end{notation}

\begin{definition}
A pair $(X, \Delta)$ is said to be \emph{log canonical} if for some (equivalently any) log resolution $\pi$ as in Notation \ref{NotationCharZeroSings}, we have that the $a_i$ that appear in Equation \ref{EqnDiscrepancy} all satisfy $a_i \geq -1$.
\end{definition}

\begin{definition} \cite{KawamataFujita}
A reduced irreducible subscheme $Z \subseteq X$ is said to be a \emph{log canonical center of $(X, \Delta)$} if there exists a log resolution $\pi : \tld X \rightarrow X$ and a divisor $E_i$ on $\tld X$ (as in Notation \ref{NotationCharZeroSings}) such that $\pi(E_i) = Z$ and such that the associated $a_i \leq -1$.
\end{definition}

\begin{remark}
A log canonical center is sometimes also called a ``center of log canonicity.''
\end{remark}

We will also need some results about seminormality.

\begin{definition} \cite{SwanSeminormality} \cite{GrecoTraversoSeminormal}
Suppose that $R$ is a reduced excellent ring and that $S \supseteq R$ is a reduced $R$-algebra which is finite as an $R$-module.  We say that the extension $i : R \subseteq S$ is \emph{subintegral} if the following two conditions are satisfied.
\begin{itemize}
\item[(1)]  $i$ induces a bijection on spectra, $\Spec S \rightarrow \Spec R$.
\item[(2)]  $i$ induces an isomorphism of residue fields over every (not necessarily closed) point of $\Spec R$.
\end{itemize}
\end{definition}

\begin{remark}
In \cite{GrecoTraversoSeminormal}, subintegral extensions are called quasi-isomorphisms.
\end{remark}

\begin{definition} \cite{SwanSeminormality} \cite{GrecoTraversoSeminormal}
Suppose that $R$ is a reduced excellent ring.  We say that $R$ is \emph{seminormal} if every subintegral extension $R \subseteq S$ is an isomorphism.
\end{definition}

\begin{remark}
In \cite{GrecoTraversoSeminormal}, the authors call $R$ seminormal if there is no proper subintegral extension $S \supseteq R$ such that $S$ is contained in the integral closure of $R$ (in its total field of fractions).  However, it follows from \cite[Corollary 3.4]{SwanSeminormality} that the above definition is equivalent.
\end{remark}

We next recall some facts about the log canonical centers of a log canonical pair.

\begin{theorem} \cite{AmbroSeminormalLocus}, \cite[Theorem 3.46]{FujinoIntroductionToLMMPForLCBook}, \cite{AmbroBasicPropertiesOfLCCenters}
\label{ThmPropertiesOfLCCenters}
Suppose that $(X, \Delta)$ is log canonical.  Then:
\begin{itemize}
\item[(a)]  The number of log canonical centers of $(X, \Delta)$ is finite.
\item[(b)]  Any intersection of two log canonical centers is a union of log canonical centers.
\item[(c)]  Any union of log canonical centers is seminormal.
\end{itemize}
\end{theorem}

Our goal in this section is to prove that the ideals associated to log canonical centers form a intersection compatible system.  %First we remind the reader that gluing along closed subschemes is possible.
%
%\begin{lemma}
%Suppose that $A$ and $B$ are rings that both map surjectively onto a ring $R$ by maps $\phi_A$ and $\phi_B$ respectively.  Let $C = \ker(\xymatrix{A \oplus B \ar[r]^-{\phi_A - \phi_B} R})$.  Note $C$ can also be viewed as the pullback of a diagram of rings.  Then $\Spec C$ is the gluing of $\Spec A$ to $\Spec B$ along the common closed subscheme $\Spec R$.  In particular:
%\begin{itemize}
%\item[(i)]  The points (closed or not) of $\Spec C$ are in bijective correspondence with the points of $(\Spec A \coprod \Spec B) / \sim$ where two points of $(\Spec A \coprod \Spec B)$ are identified if and only if they are the image of the same point in $\Spec R$.
%\item[(ii)]  $\Spec A$ and $\Spec B$ are both canonically identified with closed subschemes of $\Spec C$.  In particular, the residue fields of points of $\Spec C$ map isomorphically to the residue fields of $\Spec B$ and $\Spec C$.
%\end{itemize}
%\end{lemma}
%
%However, to do that we need the following lemmas, which may be known to experts.
We first need the following lemma, which may be known to experts.

\begin{lemma}
\label{LemmaIPlusJRadical}
Suppose that $I$ and $J$ are radical ideals in an excellent ring $R$ such that $R/(I \cap J)$ is seminormal.  Then $I + J$ is a radical ideal.
\end{lemma}
\begin{proof}
Set $K$ to be the kernel of that map
\[
 \xymatrix@R=6pt{R/I \oplus R/J \ar[r]^-{-} & R/\sqrt{I+J}\\
 (a + I, b+J) \ar@{|->}[r] & (a - b) + \sqrt{I + J}
}
\]
Note that $K$ is both an $R$-algebra and a finite $R$-submodule of $R/I \oplus R/J$.  $K$ is also reduced since $R/I$ and $R/J$ are reduced.  Geometrically, $\Spec K$ is the gluing of $\Spec R/I$ to $\Spec R/J$ along the common closed subscheme $\Spec R/\sqrt{I+J}$.  We have the following diagram of short exact sequences.
\[
\xymatrix{
0 \ar[r] & R/(I \cap J) \ar@{^{(}->}[d]_{\phi}  \ar[r]^\beta & R/I \oplus R/J \ar[d]^{\sim}  \ar[r]^-{-} & R/(I + J) \ar@{->>}[d]^{\psi} \ar[r] & 0\\
0 \ar[r] & K \ar[r] & R/I \oplus R/J \ar[r]^-{-} & R/\sqrt{I+J} \ar[r] & 0
}
\]
First, note that $\phi : \xymatrix{R/(I \cap J) \ar@{^{(}->}[r] & K}$ is automatically finite since $\beta$ is finite.  Additionally,
we claim that $\phi$ is a subintegral extension.  To see this, note $\Spec K$ is composed exactly of points of $\left( (\Spec R/I) \coprod (\Spec R/J) \right) / \sim$ where a pair of points are identified if they are both the image of the same point from $\Spec R/\sqrt{I + J}$.  This gluing operation does not change the residue fields at these points, which implies that $\phi$ is subintegral.  For additional background on gluing schemes along closed subschemes; see for example \cite{FerrandConducteur} or \cite{SchwedeGluing}.

Since $R/(I \cap J)$ is seminormal, it follows that $\phi$ is an isomorphism.  Therefore, $\psi$ is also an isomorphism and $I+ J = \sqrt{I + J}$ as desired.
\end{proof}

\begin{theorem}
\label{ThmBoundOnLCCenters}
Suppose that $(X, \Delta)$ is a log canonical pair and that $x \in X$ is a point with embedding dimension $n$.  Then the number of $d$-dimensional log canonical centers of $(X, \Delta)$ which contain $x$ is less than or equal to ${n \choose d}$.
\end{theorem}
\begin{proof}
We need to show that the set
\[
\sQ = \{ Q \in \Spec \O_{X, x} : \text{ $Q$ is the defining ideal of a log canonical center} \}
\]
is an intersection compatible system.  Suppose that $I$ and $J$ are each finite intersections of $Q \in \sQ$.  We will show that then $I+J$ is also a finite intersection of $Q \in \sQ$.  It follows from Theorem \ref{ThmPropertiesOfLCCenters}(b) that $\sqrt{I+J}$ is a finite intersection of elements of $\sQ$.  However, $\Spec R/(I \cap J)$ is a union of log canonical centers, and so it is seminormal by Theorem \ref{ThmPropertiesOfLCCenters}(c).  Therefore $I + J = \sqrt{I + J}$ is an intersection of elements of $\sQ$ by Lemma \ref{LemmaIPlusJRadical}.  Induction then implies that $\sQ$ is an intersection compatible system.
\end{proof}

\begin{remark}
\label{logcalabiyau}
 Suppose that $X$ is a normal projective variety over an algebraically closed field and that $(X, \Delta)$ is a log Calabi-Yau pair (that is, $\Delta$ is an effective $\bQ$-divisor, $(X, \Delta)$ is log canonical and $K_X + \Delta$ is $\bQ$-linearly equivalent to $0$).  For a very ample line bundle $\sL$ corresponding to a projectively normal embedding, we construct the associated section ring $S = \oplus_{i \geq 0} H^0(X, \sL^i)$.  On $Y = \Spec S$, there is an associated $\bQ$-divisor $\Delta_{Y}$ corresponding to $\Delta$.  One can show that the pair $(Y, \Delta_Y)$ is log canonical; see for example \cite[Proposition 4.38]{FujinoIntroductionToLMMPForLCBook} or \cite[Proposition 5.4(2)]{SchwedeSmithLogFanoVsGloballyFRegular}.  Furthermore, to each $d$-dimensional log canonical center of $(X, \Delta)$, there is an associated $(d+1)$-dimensional log canonical center of $(Y, \Delta_Y)$.  Thus, by Theorem \ref{ThmBoundOnLCCenters}, $(X, \Delta)$ has at most ${n \choose d+1}$ $d$-dimensional log canonical centers, where $n = \dim H^0(X, \sL)$.  This global bound is the analog of Theorem \ref{ThmCompatiblySplitProjective}.
\end{remark}

%\bibliographystyle{skalpha}
%\bibliography{NumberCompatSplit}

\providecommand{\bysame}{\leavevmode\hbox to3em{\hrulefill}\thinspace}
\providecommand{\MR}{\relax\ifhmode\unskip\space\fi MR}
\providecommand{\MRhref}[2]{%
  \href{http://www.ams.org/mathscinet-getitem?mr=#1}{#2}
}
\providecommand{\href}[2]{#2}

\end{document}